\documentclass[12pt]{amsart}
\usepackage{amsfonts,latexsym,amsthm,amssymb,xspace}
\usepackage{amscd}
\usepackage{amsmath}
\usepackage{color}
\usepackage{amsmath,amsthm,amsfonts,color}

\usepackage{amsmath,amsthm,amsfonts,amssymb}
\usepackage{array}
\theoremstyle{plain}
\newtheorem{theor}{Theorem}
\theoremstyle{plain}
\newtheorem{prop}{Proposition}
\theoremstyle{plain}
\newtheorem{lemma}{Lemma}
\theoremstyle{plain}
\newtheorem{cor}{Corollary}
\theoremstyle{remark}
\newtheorem{rem}{Remark}
\theoremstyle{remark}
\newtheorem{ex}{Example}

\begin{document}


\title{On subspaces of Orlicz spaces spanned by independent copies of a mean zero function}


\author{S.~V.~Astashkin}
\address{Samara National Research University, Samara, Russia; Lomonosov Moscow State University, Moscow, Russia; Moscow Center of Fundamental and applied Mathematics, Moscow, Russia; Bahcesehir University, Istanbul, Turkey}
\email{astash@ssau.ru}
\thanks{$^\dagger$\,This research was performed at Lomonosov Moscow State University and was supported by the Russian Science Foundation, project (no. 23-71-30001).}


\maketitle




{\bf Abstract}: We study subspaces of Orlicz spaces $L_M$ spanned by independent copies $f_k$, $k=1,2,\dots$, of a function $f\in L_M$, $\int_0^1 f(t)\,dt=0$. Any such a subspace $H$ is isomorphic to some Orlicz sequence space $\ell_\psi$. In terms of dilations of the function $f$, a description of strongly embedded subspaces of this type is obtained, and conditions, guaranteeing that the unit ball of such a subspace consists of functions with equicontinuous norms in $L_M$, are found.  In particular, we prove that there is a wide class of Orlicz spaces $L_M$ (containing $L^p$-spaces, $1\le p< 2$), for which each of the above properties of $H$ holds  if and only if the Matuszewska-Orlicz indices of the functions $M$ and $\psi$ satisfy the inequality: $\alpha_\psi^0>\beta_M^\infty$.

\def\hj{{\mathbb R}}

\def\fg{{\mathbb N}}

\vskip 0.4cm

\label{Intro}

According to the classical Khintchine inequality (see, for example, \cite[Theorem~V.8.4]{Z1}), for each $0<p<\infty$, there exist constants $A_p>0$ and $B_p>0$ such that for any sequence of real numbers $(c_k)_{k=1}^\infty$ we have 
\begin{equation}\label{ineq 1}
A_p \|(c_k)\|_{\ell^2}\leq \Big\|\sum_{k=1}^\infty c_k r_k\Big\|_{L^p[0, 1]} 
\leq B_p \|(c_k)\|_{\ell^2},
\end{equation}
where $r_k$ are the Rademacher functions, $r_k(t) = {\rm sign}
(\sin 2^k \pi t)$, $k \in {\Bbb N}$, $t \in [0,1]$, and $\|(c_k)\|_{\ell^2}:=\left(\sum_{k=1}^\infty c_k^2 \right)^{1/2}$.
This means that, for every $0<p<\infty$, the sequence $\{r_k\}_{k=1}^\infty$ is equivalent in $L^p$ to the canonical basis in the space $\ell^2$.
This example demonstrates a certain general phenomenon, which is reflected in the following concept. A closed linear subspace $H$ of the space $L^p=L^p[0,1],$ $1\le p<\infty,$ is called a {\it $\Lambda(p)$-space} if convergence in $L^p$-norm is equivalent  on $H$ to convergence in measure, or equivalently: for each (or some) $q\in (0,p)$ there is a constant $C_q>0$ such that 
\begin{equation}\label{ineq 1 extra}
 \|f\|_{L^p}\le C_q\|f\|_{L^q}\;\;\mbox{for all}\;\;f\in H
\end{equation}
(see \cite[Proposition~6.4.5]{AK}). Consequently, the inequality \eqref{ineq 1} shows that the span $[r_k]$ in $L^p$ is a $\Lambda(p)$-space for any $1\le p<\infty.$

The starting point for introducing the notion of a $\Lambda(p)$-space was the classical Rudin's paper \cite{Rud}, devoted to Fourier analysis on the circle $[0,2\pi)$, in which the following related concept was studied. Let $0<p<\infty.$ A set $E\subset \mathbb{Z}$ is called a {\it $\Lambda(p)$-set} if for some $0<q<p$ there is a constant $C_q>0 $ such that inequality \eqref{ineq 1 extra} holds for every trigonometric polynomial $f$ with spectrum
(i.e., the support of its Fourier transform) contained in $E$. As is easy to see, this is equivalent to the fact that the subspace $L_E$ spanned by the set of exponentials $\{e^{2\pi int},$ $n\in E\}$ is a $\Lambda(p)$-space.
In particular, in \cite{Rud} for all integers $n>1$, Rudin constructed
$\Lambda(2n)$-sets that are not $\Lambda(q)$-sets for any $q>2n.$ In 1989, J. Bourgain strengthened this result by extending Rudin's theorem to all $p>2$ \cite{bour}. In view of the well-known 
Vall\'{e}e Poussin criterion (see Lemma \ref{lemma 3} below), this implies, for each $p>2$, the existence of a $\Lambda(p)$-set $E$ such that  functions of the unit ball of the subspace $L_E$ fail to have equicontinuous norms in $L^p$ (for all definitions see \S\,\ref{prel}). 

On the "other side"\;of $L^2$, as it often happens, the picture turned out to be completely different.
Even earlier, in 1974, Bachelis and Ebenshtein showed in \cite{BachEb}  that in the case when $p\in (1,2)$ every $\Lambda(p)$-set is a $\ Lambda(q)$-set for some $q>p$\footnote{For a detailed exposition of the theory of $\Lambda(p)$-sets, see the survey \cite{Bourg-survey}.}. Moreover, in the same direction, in \cite[Theorem~13]{Ros} Rosenthal proved that for every $1<p<2$ a (closed linear) subspace $H$ of the space $L^p$ is a $\Lambda(p)$-space if and only if functions of the unit ball of $H$ have equicontinuous norms in $L^p$.

A recent author's paper \cite{A-22} deals with extending Rosenthal's theorem to the class of Orlicz function spaces $L_M$. Generalizing the concept of a $\Lambda(p)$-space (see \cite[Definition~6.4.4]{AK}), a subspace $H$ of an Orlicz space $L_M$ (or a symmetric space $X$) on $[0 ,1]$ will be called {\it strongly embedded} in $L_M$ (resp. in $X$) if convergence in the $L_M$-norm (resp. in the $X$-norm) on $H$ is equivalent to convergence in measure. The condition $1<p<2$ from Rosenthal's theorem in this more general setting turns into the inequality $1<\alpha_M^\infty\le \beta_M^\infty<2$ for the Matuszewska-Orlicz indices of the function $M$.
As shown in \cite{A-22}, unlike $L^p$, the last condition does not guarantee that an analogue of Rosenthal's theorem is valid in $L_M$. In particular, the norms of functions of the unit ball  of any subspace, strongly embedded in the space $L_M$ and isomorphic to some Orlicz sequence space, are equicontinuous in $L_M$ if and only if the {\it function $t^{-1/\beta_M^ \infty}$ does not belong to $L_M$} \cite[Theorem~3]{A-22}. Thus, if this condition is not fulfilled, an analogue of Rosenthal's theorem does not hold even for this special class of subspaces of Orlicz spaces.

The family of subspaces of a space $L_M$, isomorphic to Orlicz sequence spaces, includes, in particular, subspaces spanned in $L_M$ by independent copies of mean zero functions from this space (see further \S\,\ref{aux1}). The present paper is devoted mainly to a detailed study of subspaces of this type.

Note that the research related to the class of subspaces of $L^p$-spaces with a symmetric basis, spanned by sequences $\{f_k\}_{k=1}^\infty$ of independent functions, was started quite for a long time. Interest in this topic has increased after 1958, when  Kadec \cite{Kad} "put an end"\;to the solution of the well-known Banach problem proving that, for every pair of numbers $p$ and $q$ such that $1\leq p<q<2,$ a sequence $\{\xi^{(q)}_k\}_{k=1}^\infty$ of independent copies of a $q$-stable random variable $\xi^{(q)}$  spans  a subspace in $L^p$ isomorphic to $\ell^q$.
Following this, in 1969, Bretagnolle and Dacunha-Castelle showed (see \cite{BD, BD2, B3}) that for any function $f\in L^p$ such that $\int_0^1 f(t)\,dt=0$, a sequence $\{f_k\}_{k=1}^\infty$ of independent copies of $f$ is equivalent in $L^p$, $1\le p<2$, to the canonical basis in some Orlicz sequence space $\ell_\psi$, where the function $\psi$ is $p$-convex and $2$-concave. Later on, somewhat closed results were obtained by Braverman (see \cite[Corollary~2.1]{Br2} and \cite{Br3}).
In the opposite direction, as was shown in \cite{BD2}, if $\psi$ is a $p$-convex and $2$-concave Orlicz function such that $\lim_{t\to 0}\psi(t)t^{-p}=0$, then a sequence of independent copies of some mean zero function $f\in L^p$ is equivalent in $L^p$ to the canonical basis in $\ell_\psi$.

This research was then continued in the paper \cite{AS-14} due to Astashkin and Sukochev, where, among other things, the existence of direct connections between an Orlicz function $\psi$ and the distribution of a function $f\in L^p$, whose independent copies span in $L^p$ a subspace isomorphic to the space $\ell_\psi$, has been  revealed. This led to a natural question about whether the distribution of such a mean zero function $f\in L^p$ is uniquely determined (up to equivalence for large values of the argument) by a given $\psi$?  A partial solution of this problem was obtained in subsequent  papers \cite{ASZ-15} and \cite{ASZ-22}. In particular, according to \cite[Theorem~1.1]{ASZ-15}, if an Orlicz function $\psi$ is situated sufficiently "far"\;from the "extreme"\;functions $t^p$ and $t^2$, $1\leq p<2$, such uniqueness exists, and the distribution of such a function $f$ is equivalent (for large values of the argument) to the distribution of the function ${1}/{\psi^{-1}}$. In \cite{ASZ-22} some of these results were extended to general
symmetric function spaces on $[0,1]$ satisfying certain conditions.

In this paper, the above facts are used in essential way. Other important ingredients in the proofs are a version of the famous Vall\'{e}e Poussin criterion, as well as the author's results obtained in the paper \cite{A-16}, which imply that an Orlicz space $L_M$ such that $1<\alpha_M^\infty\le \beta_M^\infty<2$ contains the function $1/\psi^{-1}$ provided that there is a strongly embedded subspace in $L_M$ isomorphic to the Orlicz sequence space $\ell_\psi$.

Let us describe briefly the content of the paper. In \S\,\ref{prel} and \S\,\ref{prel1a}, we give necessary preliminary information and some auxiliary results related to symmetric spaces, as well as to Orlicz functions and Orlicz spaces.

The main results are contained in \S\,\ref{main}. Thus, in \S\,\ref{main2}, by using terms of dilations of a function $f\in L_M$, $\int_0^1 f(t)\,dt=0$, the conditions, under which the subspace $[f_k]$ spanned by independent copies of $f$ is strongly embedded in $L_M$,  are found (see Proposition \ref{prop 1a}). Here, we also obtain the conditions, ensuring that the unit ball of the subspace $[f_k]$ of the above type consists of functions having equicontinuous norms in $L_M$ (Proposition \ref{prop 2a}). In \S\,\ref{main3},  these results are applied when considering the question if the fact that the subspace $[f_k]$ is strongly embedded in $L_M$ implies the equicontinuity in $L_M$ of the norms of functions of the unit ball of this subspace (see Theorem \ref {cor: new11}).

The most complete results are obtained in \S\,\ref{main1}, when we have  $t^{-1/\beta_M^\infty}\not\in L_M$ (in particular, this   condition holds for $L^p$). Namely, if $1<\alpha_M^\infty\le \beta_M^\infty<2$ and the subspace $[f_k]$ is isomorphic to the  Orlicz sequence space $\ell_\psi$, then the above properties of this subspace can be characterized by using the Matuszewska-Orlicz  indices of the functions $M$ and $\psi$ as follows: the unit ball of the subspace $[f_k]$ consists of functions having equicontinuous norms in $L_M$ $\Longleftrightarrow$ the subspace $[f_k]$ is strongly embedded in $L_M$ $\Longleftrightarrow $ $\alpha_\psi^0>\beta_M^\infty$ (see Theorem \ref{new theorem}).

In the final part of the paper, \S\,\ref{main4}, it is shown that the unit ball of any subspace of the $L^2$-space spanned by mean zero identically distributed independent functions consists of functions with equicontinuous norms in $L ^2$ (see Theorem \ref{prop 5}).

Some of the results of this paper were announced in the note \cite{A-23DAN}.

\section{Preliminaries.}
\label{prel}

If $F_1$ and $F_2$ are two non-negative functions (quasinorms) defined on a set $T$, then the notation $F_1\preceq F_2$ means the existence of a constant $C>0$ such that $F_1(t)\leq CF_2( t)$ for all $t\in T$. If simultaneously $F_1\preceq F_2$ and $F_2\preceq F_1$, the quantities $F_1$ and $F_2$ will be called {\it equivalent} on $T$ (we write: $F_1\asymp F_2$). In the case when $T=(0,\infty)$, we will also say about the equivalence {\it for large (resp. small) values of the argument}. This means that the relation $F_1\asymp F_2$ holds for all $t\ge t_0$ (resp. $0<t\le t_0$), where $t_0$ is sufficiently large (resp. sufficiently small).

The fact that Banach spaces $X$ and $Y$ are linearly and continuously isomorphic will be denoted as $X\approx Y$. A subspace of a Banach space always will be assumed to be linear and closed. Finally, in what follows $C$, $C_1,\dots$ are positive constants, the value of which can change from case to case.

\subsection{Symmetric spaces.\\}
\label{prel1}

For a detailed exposition of the theory of symmetric spaces, see the monographs \cite{KPS,LT,BSh}.

A Banach space $X$ of real-valued functions measurable on the space $(I,m)$, where $I=[0,1]$ or $(0,\infty)$ and $m$ is the Lebesgue measure, is called {\it symmetric} (or {\it rearrangement invariant}), if from the conditions $y\in X$ and $x^*(t)\le y^*(t)$ almost everywhere (a.e.) on $I$ it follows: $x\in X$ and ${\|x\|}_X \le {\|y\|}_X.$ Here and throughout, $x^*(t)$ denotes right-continuous nonincreasing {\it rearrangement} of a function $|x(s)|$, given by:
$$
x^{*}(t):=\inf \{ \tau\ge 0:\,n_x(\tau)\le t \},\;\;0<t<m(I),$$
where
$$
n_x(\tau):=m\{s\in I:\,|x(s)|>\tau\},\;\;\tau>0.$$

In particular, every symmetric space $X$ is a Banach lattice of measurable functions, which means the following: if $x$ is measurable on $I$, $y \in X$ and $|x(t)|\le |y(t) |$ a.e. on $I$, then $x\in X$ and ${\|x\|}_X \le {\|y\|}_X.$
Moreover, according to the definition, if $x$ and $y$ are {\it equimeasurable functions}, i.e., $n_x(\tau)=n_y(\tau)$ for all $\tau>0$, and $y\in X$, then $x\in X$ and ${\|x\|}_X ={\ |y\|}_X.$ Note that every measurable function $x(t)$ is equimeasurable with its rearrangement $x^*(t)$.

For each symmetric space $X$ on $[0,1]$ (resp. on $(0,\infty)$) we have the continuous embeddings $L^\infty[0,1]{\subseteq} X{\subseteq } L^1[0,1]$ (resp. $(L^1\cap L^\infty)(0,\infty){\subseteq} X {\subseteq} (L^1+L^\infty) (0,\infty)).$
In what follows, it will be assumed that the normalization condition  $\|\chi_{[0,1]}\|_X=1$ is satisfied. In this case, the constant in each of the preceding embeddings is equal to $1$.

The {\it fundamental function} $\phi_X$ of a symmetric space $X$ is defined by the formula $\phi_X(t):=\|\chi_A\|_X$, where $\chi_A$ is the characteristic function of a measurable set $A \subset I$ such that $m(A)=t$. The function $\phi_X$ is {\it quasi-concave} (i.e., $\phi_X(0)=0$, $\phi_X$ does not decrease and $\phi_X(t)/t$ does not increase on $I$).

Let $X$ be a symmetric space on $[0,1]$. For any $\tau>0$ the {\it dilation operator} ${\sigma}_\tau x(t):=x(t/\tau)\chi_{(0,\min\{1,\tau\})}(t)$, $0\le t\le 1$, is bounded in $X$ and $\|{\sigma}_\tau\|_{X\to X}\le \max(1,\tau)$ (see, e.g., \cite[Theorem II.4.4]{KPS}). To avoid any confusion, we will not introduce a special notation for the dilation operator $x(t)\mapsto x(t/\tau)$, $\tau>0$, defined on the set of functions $x(t)$ measurable on $(0.\infty)$. The norm of this operator in any symmetric space $X$ on the semi-axis satisfies exactly the same estimate as the norm of the above operator ${\sigma}_\tau$.

If $X$ is a symmetric space on $[0,1]$, then the {\it associated} space $X'$ consists of all measurable functions $y$, for which
$$
\|y\|_{X'}:=\sup\Big\{\int_{0}^1{x(t)y(t)\,dt}:\;\;
\|x\|_{X}\,\leq{1}\Big\}<\infty.
$$
$X'$ is also a symmetric space; it is isometrically embedded in the dual space $X^*$, and $X'=X^*$ if and only if $X$ is separable. A symmetric space $X$ is called {\it maximal} if, from the conditions $x_n\in X$, $n=1,2,\dots,$ $\sup_{n=1,2,\dots}\|x_n\|_X<\infty$ and $x_n\to{x}$ a.e., it follows that $x\in X$ and $||x||_X\le \liminf_{n\to\infty}{||x_n||_X}.$ The space $X$ is maximal if and only if the canonical embedding of $X$ in its second associated $X''$ is an isometric surjection.

In a similar way, one can also define symmetric sequence spaces (see, for instance, \cite[\S\,II.8]{KPS}). In particular, if $X$ is a symmetric sequence space, then the {\it fundamental function} of $X$ is defined by the formula $\phi_{X}(n):=\|\sum_{k=1}^n e_k\| _{X}$, $n=1,2,\dots$. In what follows, $e_k$ are canonical unit vectors in sequence spaces, i.e., $e_k=(e_k^i)_{i=1}^\infty$, $e_k^i=0$, $i\ne k$, and $e_k^k=1$, $k,i=1,2,\dots$.

The family of symmetric spaces includes many classical spaces that play an important role in analysis, in particular, $L^p$-spaces, Orlicz, Lorentz, Marcinkiewicz spaces and many others.
The next part of this section contains some used further preliminaries from the theory of Orlicz spaces, which are the main subject of the study in this paper.

\subsection{Orlicz functions and Orlicz spaces.\\}
\label{prel2}

Orlicz spaces are the most natural and important generalization of $L^p$-spaces. A detailed exposition of their properties can be found in the monographs \cite{KR,RR,Mal}.

Let $M$ be an Orlicz function, i.e., an increasing, convex, continuous function on the semi-axis $[0, \infty)$ such that $M(0) = 0$. Without loss of generality, we assume throughout the paper that $M(1) = 1$. The {\it Orlicz space} $L_{M}:=L_M(I)$ consists of all functions $x(t)$ measurable on $I$, for which the Luxemburg norm 
$$
\| x \|_{L_{M}}: = \inf \left\{\lambda > 0 \colon \int_I M\Big(\frac{|x(t)|}{\lambda}\Big) \, dt \leq 1 \right\}
$$
is finite. In particular, if $M(u)=u^p$, $1\le p<\infty$, we obtain the space $L^p$ with the usual norm.

Note that the definition of the space $L_M[0,1]$ depends (up to equivalence of norms) only on the behaviour of the function $M(u)$ for large values of $u$. The fundamental function of this space can be calculated by the formula $\phi_{L_M}(u)=1/M^{-1}(1/u)$, $0<u\le 1$, where $M^{-1}$ is the inverse function for $M$.

If $M$ is an Orlicz function, then the {\it complementary} (or {\it Yang conjugate}) function $\tilde{M}$ for $M$ is defined as follows:
$$
\tilde{M}(u):=\sup_{t>0}(ut-M(t)),\;\;u>0.$$
As is easy to see, $\tilde{M}$ is also an Orlicz function, and the  complementary function for $\tilde{M}$ is $M$.

Every Orlicz space $L_M(I)$ is maximal; $L_M[0,1]$ (resp. $L_M(0,\infty)$) is separable if and only if the function $M$ satisfies the so-called {\it $\Delta_2^\infty$-condition} ($M \in \Delta_2^\infty$) (resp. {\it $\Delta_2$-condition} ($M\in \Delta_2$) ), i.e.,
$$
\sup_{u\ge 1}{M(2u)}/{M(u)}<\infty\;\;\mbox{(resp.}\;\sup_{u>0}{M(2u)}/{M(u)}<\infty).$$ 
In this case, $L_M(I)^*=L_M(I)'=L_{\tilde{M}}(I)$.

An important characteristic of an Orlicz space $L_M[0,1]$ are {\it Matuszewska-Orlicz indices at infinity} $\alpha_{M}^{\infty}$ and $\beta_{M}^{\infty}$, defined by
$$
\alpha_{M}^{\infty}: = \sup \big\{ p : \sup_{t, s \geq 1} \frac{M(t)s^{p}}{M(ts)} < \infty \big\}, \ \ \ \beta_{M}^{\infty}: = \inf \big\{ p : \inf_{t, s \geq 1} \frac{M(t)s^{p}}{M(ts)} > 0 \big\}
$$
(see \cite{LTIII} or \cite[Proposition~5.3]{KamRy}). It can be easily checked that $1 \leq \alpha_{M}^{\infty} \leq \beta_{M}^{\infty} \leq \infty$.
Moreover, $M\in \Delta_2^\infty$ (resp. $\tilde{M}\in \Delta_2^\infty$) if and only if $\beta_{M}^{\infty} <\infty$ (resp. $\alpha_{M}^{\infty}>1$).

The Matuszewska-Orlicz indices are being a special case of the so-called Boyd indices, which can be defined for any symmetric space on $[0,1]$ or $(0,\infty)$ (see, e.g., \cite[Definition~ 2.b.1]{LT} or \cite[\S\,II.4, p.~134]{KPS}).

\vskip0.2cm

Similarly, one can define an {\it Orlicz sequence space}. Namely, if $\psi$ is an Orlicz function, then the space $\ell_{\psi}$ consists of all sequences $a=(a_{k})_{k=1}^{\infty}$ such that
$$
\| a\|_{\ell_{\psi}} := \inf\left\{\lambda>0: \sum_{k=1}^{\infty} \psi \Big( \frac{|a_{k}|}{\lambda} \Big)\leq 1\right\}<\infty.
$$
If $\psi(u)=u^p$, $p\ge 1$, we have $\ell_\psi=\ell^p$ isometrically. 

The fundamental function of an Orlicz space $\ell_{\psi}$ may be calculated by using the formula 
\begin{equation}
\label{fund}
\phi_{{\ell_\psi}}(n)=\frac{1}{\psi^{-1}(1/n)},\;\;n=1,2,\dots
\end{equation}

A space $\ell_{\psi}$ is separable if and only if $\psi$ satisfies the $\Delta_{2}^{0}$-condition ($\psi\in \Delta_{2}^{0}$), that is,
$$
\sup_{0<u\le 1}{\psi(2u)}/{\psi(u)}<\infty.$$
In this case $\ell_{\psi}^*=\ell_{\psi}'=\ell_{\tilde{\psi}}$, where  $\tilde{\psi}$ is the complementary function for $ \psi$.

As is easy to check (see also \cite[Proposition~4.a.2]{LT77}), the unit vectors $e_n$, $n=1,2,\dots$, form a symmetric basis in any Orlicz sequence space $\ell_{\psi}$ if $\psi\in \Delta_{2}^{0}$. Recall that a basis $\{x_n\}_{n=1}^\infty$ of a Banach space $X$ is called {\it symmetric} if there exists a constant $C>0$ such that for an arbitrary permutation $\pi$ of the set of positive integers and any $a_n\in\mathbb{R}$ the following inequality holds:
$$
C^{-1}\Big\|\sum_{n=1}^{\infty}a_nx_n\Big\|_X\le \Big\|\sum_{n=1}^{\infty}a_nx_{\pi(n)}\Big\|_X\le C\Big\|\sum_{n=1}^{\infty}a_nx_n\Big\|_X.$$

The definition of an Orlicz sequence space $\ell_{\psi}$ depends (up to equivalence of norms) only on the behaviour of the function $\psi$ for small values of the argument. More precisely, if $\varphi,\psi \in \Delta_{2}^{0}$, then the following conditions are equivalent: (1) $\ell_{\psi}=\ell_{\varphi}$ (with equivalence of norms); 2) the canonical vector bases in the spaces $\ell_{\psi}$ and $\ell_{\varphi}$ are equivalent; 3) the functions $\psi$ and $\varphi$ are equivalent for small values of the argument
(see \cite[Proposition~4.a.5]{LT77} or \cite[Theorem~3.4]{Mal}). In the case when $\psi$ is a {\it degenerate} Orlicz function, i.e., $\psi(u)=0$ for some $u> 0$, we obtain that $\ell_{\psi}=\ell_\infty$ (with equivalence of norms).

Let $\psi$ be an Orlicz function, $\psi\in \Delta_{2}^{0}$, $A>0$. We define the following subsets of the space $C[0,1]$:
$$
E_{\psi, A}^{0} = \overline{\big\{ {\psi(st)}/{\psi(s) \ : 0<s<A \big\}}}, 
\ \ C_{\psi, A}^{0} = \overline{conv E_{\psi, A}^{0} }, 
$$
where the closure is taken in the $C[0,1]$-norm, and $conv F$ denotes the convex hull of a set $F\subset C[0,1]$. All these sets are non-empty compact subsets of the space $C[0,1]$ \cite[Lemma~4.a.6]{LT77}. According to the theorem due to Lindenstrauss and Tsafriri (see, e.g., \cite[Theorem~4.a.8]{LT77}), an Orlicz space $\ell_\varphi$ is isomorphic to some subspace of the space ${\ell_\psi}$ if and only if $\varphi\in C_{\psi,1}^{0}$.

For any Orlicz function $\psi$ we define the {\it Matuszewska-Orlicz indices at zero} $\alpha_{\psi}^{0}$ and $\beta_{\psi}^{0}$ by
$$
\alpha_{\psi}^{0}: = \sup \big\{ p : \sup_{0<t, s \leq 1} \frac{\psi(st)}{s^{p}\psi(t)} < \infty \big\}, \ \ \ \beta_{\psi}^{0}: = \inf \big\{ p : \inf_{0<t, s \leq 1} \frac{\psi(st)}{s^{p}\psi(t)} > 0 \big\}
$$
As for the Matuszewska-Orlicz indices at infinity, the following inequalities hold: $1 \leq \alpha_{\psi}^{\infty} \leq \beta_{\psi}^{\infty} \leq \infty$ (see, e.g., \cite[Chapter~4]{LT77}). Moreover, the space $\ell^p$ or $c_0$ if $p=\infty$ is isomorphic to some subspace of an Orlicz space $\ell_\psi$ if and only if $p\in [\alpha_{\psi }^{0},\beta_{\psi}^{0}]$ \cite[Theorem~4.a.9]{LT77}.

\vskip0.2cm




\section{Auxiliary results.}
\label{prel1a}

\subsection{Strongly embedded subspaces and sets of functions with equicontinuous norms.\\}
\label{prel3}

Let $X$ be a symmetric space on $[0,1]$. Recall (see \S\,\ref{Intro}) that a subspace $H\subset X$ is {\it strongly embedded} if convergence in the $X$-norm on $H$ is equivalent to convergence in measure.

The following result is known in one form or another (for the case of $L^p$-spaces see \cite[Proposition~6.4.5]{AK}). For the reader's convenience, we present here its proof.

\begin{lemma}\label{lemma 2dop}
Suppose $X$ is a symmetric space on $[0,1]$ such that $X\ne L^1$ and  $H$ is a subspace of $X$. If the norms of $X$ and $L^1$ are equivalent on $H$, then $H$ is strongly embedded in $X$.
\end{lemma}
\begin{proof}
Assuming the contrary, we find a sequence $\{x_n\}\subset X$ such that
$\{x_n\}$ converges to zero in measure, but $\|x_n\|_X\not\to 0$. Passing to a subsequence, we can take for granted that $\{x_n\}$ converges to zero a.e. on $[0,1]$ and $\|x_n\|_X=1$, $n=1,2,\dots$. Then, for any $A>0$
\begin{eqnarray}
\|x_n\|_{L^1} &=& \int_{\{|x_n|\ge A\}} |x_n(t)|\,dt+\int_{\{|x_n|<A\}} |x_n(t)|\,dt\nonumber\\ &\le& \|x_n\|_{X}\|\chi_{\{|x_n|\ge A\}}\|_{X'}+\int_{\{|x_n|<A\}} |x_n(t)|\,dt\nonumber\\ &=& \phi_{X'}(m\{|x_n|\ge A\})+\int_{\{|x_n|<A\}} |x_n(t)|\,dt,
\label{2dop}
\end{eqnarray}
where $X'$ is the associated space for $X$ and $\phi_{X'}$ is the fundamental function of $X'$ (see \S\,\ref{prel1}). From $X\ne L^1$ it follows $X'\ne L_\infty$, and, as one can easily check,   $\lim_{u\to 0+}\phi_{X'}(u)=0$.

Let $\delta>0$ be arbitrary. First, for all $n=1,2,\dots$ we have
$$
m\{|x_n|\ge A\}\le \frac{\|x_n\|_{L^1}}{A}\le \frac{\|x_n\|_X}{A}=
\frac{1}{A},$$
and, consequently, there is $A_0>0$ such that
$$
\sup_{n=1,2,\dots}\phi_{X'}(m\{|x_n|\ge A_0\})\le\frac{\delta}{2}.$$
Second, by the Lebesgue Dominated Convergence theorem, there is a positive integer $n_0$ such that, for the above $A_0$ and for all $n\ge n_0$ it holds
$$
\int_{\{|x_n|<A_0\}} |x_n(t)|\,dt\le\frac{\delta}{2}.$$
As a result, applying the last two inequalities, as well as estimate \eqref{2dop} for $A=A_0$, we obtain that $\|x_n\|_{L^1}\le\delta$ for $n\ge n_0$. Since $\delta>0$ is arbitrary, it follows that the norms of the spaces $X$ and $L^1$ are not equivalent on $H$. Since this contradicts the assumption, the lemma is proven.
\end{proof}

Let $X$ be a symmetric space on $[0,1]$. The functions of a set $K\subset X$ are said to have {\it equicontinuous norms} in $X$ if
$$
\lim_{\delta\to 0}\sup_{m(E)<\delta}\sup_{x\in K}\|x\chi_{E}\|_X=0.$$
Suppose $H$ is a subspace of $X$. In what follows, by $B_H$ we denote the closed unit ball of $H$, i.e., $B_H:=\{x\in H:\,\|x\|_X\le 1\}$. 

\begin{lemma}\label{prop 1ab}
Let $X$ be a symmetric space on $[0,1]$ and $H$ be a subspace of $X$, $X\ne L^1$. If the $X$-norms of functions of the set $B_H$ are equicontinuous, then $H$ is strongly embedded in $X$.
\end{lemma}
\begin{proof}
First, by the assumption and the definition of the rearrangement $x^*$, for every $\varepsilon>0$ there exists $\delta>0$ such that for any function $x\in H$, $\|x\|_X\le 1$, we have
\begin{equation}\label{eq0a}
\|x^*\chi_{[0,\delta]}\|_X\le\varepsilon.
\end{equation}

Next, for an arbitrary measurable function $x(t)$ on $[0,1]$ and each $\delta>0$ we define the set
$$
Q_x(\delta):=\{t\in [0,1]:\,|x(t)|\ge\delta\|x\|_{X}\}.$$ 
Let us show that, if $\delta>0$ is sufficiently small, the following embedding holds:
\begin{equation}\label{eq11aa}
H\subset \{x\in L^1:\, m(Q_x(\delta))\ge \delta\}.
\end{equation}
Indeed, assuming that this is not the case, for each $\delta>0$ we find a function $x_\delta\in H$ such that $m(Q_{x_\delta}(\delta))<\delta$. Then, by the definition of the rearrangement $x_\delta^*$ and the equality $\|\chi_{[0,1]}\|_X=1$, we obtain
\begin{eqnarray*}
\|x_\delta^*\chi_{[0,\delta]}\|_X &\ge& \|x_\delta^*\chi_{[0,m(Q_{x_\delta}(\delta))]}\|_X\ge \|x_\delta\chi_{Q_{x_\delta}(\delta)}\|_X\\ &\ge& \|x_\delta\|_X- \|x_\delta\chi_{[0,1]\setminus Q_{x_\delta}(\delta)}\|_X\\ &\ge& \|x_\delta\|_X- \delta\|x_\delta\|_X
\|\chi_{[0,1]}\|_X\\ &=& (1-\delta)\|x_\delta\|_X.
\end{eqnarray*}
Since $\delta>0$ and $\varepsilon>0$ are arbitrary, the last inequality contradicts \eqref{eq0a} if we take for $x$ in this inequality the function $x_\delta/\|x_\delta\|_{L_M}$ for sufficiently small $\delta$. Thus, \eqref{eq11aa} is proved.

Now, let $\delta>0$ be such that \eqref{eq11aa} holds. Then for all $x\in H$ we have
$$
\|x\|_{L^1}\ge \int_{Q_x(\delta)} |x(t)|\,dt\ge \delta\|x\|_Xm(Q_x(\delta))\ge\delta^2 \|x\|_X.$$
Since the opposite inequality $\|x\|_{L^1}\le \|x\|_X$, $x\in X$, is fulfilled for any symmetric space $X$ (see \S\,\ref {prel1}), we conclude that the norms of $X$ and $L^1$ are equivalent on $H$. The required statement now follows from Lemma \ref{lemma 2dop}.
\end{proof}

\begin{rem}
Slightly modifying the proof, one can show that Lemma \ref{prop 1ab} is valid for $X=L^1$ as well. At the same time, the converse statement to this lemma, in general, does not hold (see Remark \ref{rem: non-equi-int} below or in more detail \cite[Example~2]{A-14}).
\end{rem}

\subsection{$P$-convex and $p$-concave Orlicz functions and Matuszewska-Orlicz indices.\\}
\label{aux2}
Let $1\leq p<\infty$. An Orlicz function $M$ is said to be $p$-{\it  convex} (resp. $p$-{\it concave}) if the mapping $t \mapsto M(t^{1/p})$ is convex (resp. concave). It is easy to check that an Orlicz space $L_M[0,1]$ is $p$-convex (resp. $p$-concave) if and only if the function $M$ is equivalent to some $p$-convex (resp. $p$-concave) Orlicz function for large values of the argument.
Similarly, an Orlicz sequence space $\ell_\psi$ is $p$-convex (resp. $p$-concave) if and only if the function $\psi$ is equivalent to some $p$-convex (resp. $p$-concave) Orlicz function for small values of the argument. Recall that a Banach lattice $X$ is called $p$-{\it convex} (resp. $p$-{\it concave}), where $1 \leq p \le\infty$, if there exists $C>0$ such that for any $n\in\mathbb{N}$ and arbitrary elements $x_{1}, x_{2}, \dots, x_{n}$ from $X$
$$
\Big\|\Big(\sum_{k=1}^n |x_k|^p\Big)^{1/p}\Big\|_X \leq C \Big(\sum_{k=1}^n \|x_k\|_X^p\Big)^{1/p}
$$
(resp.
$$
\Big(\sum_{k=1}^n\|x_k\|_X^p\Big)^{1/p} \leq C \Big\| \Big(\sum_{k=1}^n |x_k|^p\Big)^{1/p}\Big\|_X)
$$
(with the natural modification of expressions in the case of $p=\infty$). Obviously, every Banach lattice is $1$-convex and $\infty$-concave with constant $1$. Moreover, the space $L^p$ is $p$-convex and $p$-concave with constant $1$.

From the definition of Matuszewska-Orlicz indices and Lemma 20 from \cite{MSS} (see also \cite[Lemma~5]{AS-14}) we obtain the following characterization of the above properties.

\begin{lemma}
\label{Lemma 20}
Let $1\leq p<\infty$ and let $\psi$ be an Orlicz function on $[0,\infty)$. Then, we have

(i) $\psi$ is equivalent to a $p$-convex (resp. $p$-concave) function for small values of the argument $\Longleftrightarrow$ $\psi(st)\le C s^{p}\psi(t)$ (resp. $s^ p\psi(t)\le C \psi(st)$) for some $C>0$ and all $0<t,s\leq 1$;

(ii) $\psi$ is equivalent to a $(p+\varepsilon)$-convex (resp. $(p-\varepsilon)$-concave) function for small values of the argument and some $\varepsilon>0$ $\Longleftrightarrow$ $\alpha_\psi^0>p$ (resp. $\beta_\psi^0<p$).
\end{lemma}

The proof of the following technical result is analogous to the proof of Lemma 6 given in \cite{A-22} and hence we skip it.

\begin{lemma}
\label{Lemma new1}
Let $\psi$ and $\varphi$ be Orlicz functions, $\varphi\in C_{\psi,1}^0$. Then, we have $\alpha_\psi^0\le \alpha_\varphi^0\le \beta_\varphi^0\le \beta_\psi^0$.
\end{lemma}

The following lemma is being a direct consequence of the results proved in \cite{A-16}.

\begin{lemma}
\label{prop1}
Let $M$ be an Orlicz function, $1<\alpha_M^\infty\le \beta_M^\infty<2$. Suppose that $H$ is a strongly embedded subspace of the Orlicz space $L_M$ such that $H\approx\ell_\psi$, where $\beta_\psi^0<2$. Then, if $\varphi\in C_{\psi,1}^0$, then $1/\varphi^{-1}\in L_M$.

In particular, we get that $t^{-1/\alpha_\psi^0}\in L_M$. Therefore, if $t^{-1/\beta_M^\infty}\not\in L_M$, then $\alpha_\psi^0>\beta_M^\infty$.
\end{lemma}

\begin{proof}
First of all, $\ell^{\alpha_\psi^0}$ is isomorphic to some subspace of the Orlicz space $\ell_\psi$ (see \cite[Theorem~4.a.9]{LT77} or 
\S\,\ref{prel2}). Consequently, by the assumption, $L_M$ contains a subspace isomorphic to $\ell^{\alpha_\psi^0}$. On the other hand, since $1<\alpha_M^\infty\le \beta_M^\infty<2$, we have $L_M\in\Delta_2^\infty$ and $L_M^*=L_{\tilde{M}}\in\Delta_2^\infty$ (see \S\,\ref{prel2}). Hence, the spaces $L_M$ and $L_M^*$ are maximal and separable. Then, by the well-known Ogasawara theorem (see, e.g., \cite[Theorem~X.4.10]{KA}), $L_M$ is reflexive. Therefore, $L_M$ does not contain subspaces isomorphic to $\ell^1$, whence $\alpha_\psi^0>1$. Thus, from the condition and Lemma \ref{Lemma new1} it follows that $1<\alpha_\varphi^0\le \beta_\varphi^0<2$.

Further, applying Lemma \ref{Lemma 20}, we obtain that, if $\varepsilon>0$ is sufficiently small, then the function $\varphi$ is $(1+\varepsilon)$-convex and $(2-\varepsilon)$-concave for small values of the argument. Moreover, since $\varphi\in C_{\psi,1}^0$, by \cite[Theorem~4.a.8]{LT77} (see also \S\,\ref{prel2}), the space $\ell_{\varphi}$ is isomorphic to some subspace of the space $\ell_\psi$. Thus, $L_M$ contains a strongly embedded subspace isomorphic to $\ell_{\varphi}$, and we can apply Corollary~3.3 from \cite{A-16} to conclude that $1/\varphi^{-1}\in L_M$.

To prove the second statement of the lemma, note that the function $\varphi(t)=t^{\alpha_\psi^0}$ belongs to the set $C_{\psi,1}^0$ (see \S\,\ref{prel2}). Therefore, as was proven, $t^{-1/\alpha_\psi^0}\in L_M$. Hence, if additionally $t^{-1/\beta_M^\infty}\not\in L_M$, then it follows that $\alpha_\psi^0>\beta_M^\infty$.
\end{proof}

\subsection{A version of Vall\'{e}e Poussin's criterion.\\}
\label{aux3}

The following simple fact will be used below.

\begin{lemma}\label{lemma element}
Let $N$ be an increasing, continuous function on the half-axis $[0, \infty)$ such that $N(u)/u$ increases for $u>0$ and $N(0) = 0$. Then, if $N\in \Delta_2$ (resp. $N\in \Delta_2^\infty$), then $N$ is equivalent to the Orlicz function $M$, defined by $M(t)=\int_0^t N(u)\,{ du}/{u}$ if $t>0$ and $M(0)=0$, on $[0,\infty)$ (resp. for large values of the argument).
\end{lemma}
\begin{proof}
Assume that $N\in \Delta_2$ (the case, when $N\in \Delta_2^\infty$ can be treated in the same way).

Note that $M$ is an increasing, continuous function on the half-axis $[0, \infty)$. Moreover, since the function $M'(t)=N(t)/t$ is increasing, then $M$ is an Orlicz function and $M(t)\le N(t)$, $t>0$. The opposite estimate follows from the condition $N\in \Delta_2$:
$$
M(t)\ge\int_{t/2}^t N(u)\frac{du}{u}\ge N(t/2)\ge K^{-1}N(t),\;\;t>0,$$
where $K$ is the $\Delta_2$-constant of $N$. Thus, $M$ and $N$
are equivalent on $[0,\infty)$, and the proof is completed.
\end{proof}

The proof of the following statement, which is a variant of the famous Vall\'{e}e Poussin's criterion (see, e.g., \cite{Al-94}, \cite{CFMN}, \cite{LMT}) can be found in the paper \cite{A-22}.

\begin{lemma}\label{lemma 3}
Let $M$ be an Orlicz function such that $M\in \Delta_2^\infty$ and $\tilde{M}\in \Delta_2^\infty$. For any $f\in L_M$ there is a function $N,$ equivalent to some Orlicz function for large values of the argument and satisfying the following conditions: $N(1)=1,$ $N\in \Delta_2^\infty$, $\tilde {N}\in \Delta_2^\infty$,
\begin{equation*}\label{eq10}
\lim_{u\to\infty}\frac{N(u)}{M(u)}=\infty
\end{equation*}
and
\begin{equation*}\label{eq11}
\int_0^1N(|f(t)|)\,dt<\infty.
\end{equation*}
Moreover, if in addition $M$ is $p$-convex for large values of the argument, then, along with the preceding properties, $N$ is also equivalent to some $p$-convex Orlicz function for large values of the argument.
\end{lemma}

\subsection{A description of subspaces of Orlicz spaces generated by   mean zero identically distributed independent functions.\\}
\label{aux1}

Recall (see, for instance, \cite[Chapter~2]{KS}) that a set of functions $\{f_k\}_{k=1}^n$, measurable on $[0,1]$, is called
{\it independent} if for any intervals $I_k\subset \hj$ we have
$$
m\{t\in [0,1]:\,f_k(t)\in
I_k,\,k=1,2,\dots,n\}=\prod_{k=1}^n m\{t\in [0,1]:\,f_k(t)\in I_k\}.$$ 
It is said that $\{f_k\}_{k=1}^\infty$ is a {\it sequence of independent functions} if the set $\{f_k\}_{k =1}^n$ is independent for each $n\in\fg$.

Let $M$ be an Orlicz function, $M\in \Delta_2^\infty$, $L_M=L_M[0,1]$ be the Orlicz space, $\{f_k\}_{k=1}^ \infty$ be a sequence of mean zero independent functions, equimeasurable with a function $f\in L_M$. Then (see \cite[p.~794]{JS} or \cite{AS-08}), with equivalence constants independent of $a_k\in\mathbb{R}$, $k=1,2,\dots$, we have
$$
\Big\|\sum_{k=1}^\infty a_kf_k\Big\|_{L_M}\asymp \Big\|\Big(\sum_{k=1}^\infty a_k^2f_k^2\Big)^{1/2}\Big\|_{L_M}.$$
In turn, if $\theta(u)=u^2$ for $0\le u\le 1$, $\theta(u)=M(u)$ for $u\ge 1$ and $\ell_\psi$ is the Orlicz sequence space, generated by the function
\begin{equation}\label{psi}
\psi(u):=\int_0^1\theta(u|f(t)|)dt,\;\;u\ge 0,
\end{equation}
then, according to  \cite[Theorem~8]{Mont-02}, it holds
$$
\Big\|\Big(\sum_{k=1}^\infty a_k^2f_k^2\Big)^{1/2}\Big\|_{L_M}\asymp \|(a_k)\|_{{\ell_\psi}}.
$$
Hence,
\begin{equation}
\Big\|\sum_{k=1}^\infty a_kf_k\Big\|_{L_M}\asymp\|(a_k)\|_{{\ell_\psi}},
\label{Luxem}
\end{equation}
which means that the sequence $\{f_k\}_{_{k=1}}^\infty$ is equivalent in $L_M$ to the canonical basis $\{e_k\}_{_{k=1}}^\infty$ in the Orlicz sequence space ${\ell_\psi}$, where $\psi$ is defined by \eqref{psi}.

Observe that, in general, $\theta$ is not an Orlicz function. However, the function $\theta(t)/t$ is increasing, continuous and from the condition $M\in \Delta_2^\infty$ it follows $\theta\in\Delta_2$. Therefore, by Lemma \ref{lemma element}, $\theta$ is equivalent on $(0,\infty)$ to the Orlicz function $\tilde{\theta}(t):=\int_0^t{\theta(u) }/{u}\,du.$ This and \eqref{psi} imply that $\psi$ is also equivalent to some Orlicz function.

Next, for every measurable function $x(t)$ on $[0,1]$ and any sequence $a = (a_k)_{k=1}^\infty$ of reals we set
$$
(a \bar\otimes x)(s):=\sum_{k=1}^\infty a_kx(s-k+1)\chi_{(k,k+1)}(s),\;\;s>0.$$
As is easy to see, the distribution function of the function $a \bar\otimes x$ is equal to the sum of the distribution functions of the terms  $a_kx$, $k=1,2,\dots$:
$$
n_{a\bar\otimes x}(\tau)= \sum_{k=1}^\infty n_{a_k x}(\tau),\quad \tau>0.$$

As above, suppose that $M$ is an Orlicz function, $\{f_k\}_{k=1}^\infty$ be a sequence of mean zero independent functions, equimeasurable with some function $f\in L_M$. According to the well-known Johnson-Schechtman theorem \cite[Theorem~1]{JS}, with constants that do not depend on $a_k\in\mathbb{R}$, $k=1,2,\dots$, we have
\begin{equation*}
\label{JoSh}
\Big\|\sum_{k=1}^\infty a_kf_k\Big\|_{L_M}\asymp \|(a \bar\otimes f)^*\chi_{[0,1]}\|_{L_M}+\|(a \bar\otimes f)^*\chi_{[1,\infty)}\|_{L^2}.
\end{equation*}
Combining this together with \eqref{Luxem}, we obtain
\begin{equation}
\|(a_k)\|_{{\ell_\psi}}\asymp \|(a \bar\otimes f)^*\chi_{[0,1]}\|_{L_M}+\|(a \bar\otimes f)^*\chi_{[1,\infty)}\|_{L^2}.
\label{Luxem1}
\end{equation}
In particular, the function 
$$
\Big(\Big(\sum_{k=1}^n e_k\Big)\bar\otimes f\Big)(s)=\sum_{k=1}^n f(s-k+1)\chi_{(k,k+1)}(s)$$
is equimeasurable with the function $f(t/n)$, $t>0$.
Thus, if $f=f^*$, then, taking into account that the fundamental function $\phi_{\ell_\psi}$ satisfies \eqref{fund} (see \S\,\ref{prel2}), by \eqref{Luxem1} and the definition of the dilation operator $\sigma_\tau$ (see \S\,\ref{prel1}), we get
\begin{eqnarray}
\frac{1}{\psi^{-1}(1/n)}&\asymp& \|\sigma_nf\|_{L_M}+\|f(\cdot/n)\chi_{[1,\infty)}\|_{L^2}\nonumber\\
&=&
\|{\sigma}_nf\|_{L_M}+\Big(n\int_{1/n}^1f(s)^2\,ds\Big)^{1/2},\;\;n\in\mathbb{N}.
\label{Luxem2}
\end{eqnarray}

Let us illustrate the above discussion with two examples, showing that the studied properties of the subspace $[f_k]:=[f_k]_{L_M}$, spanned by a sequence of independent copies of a mean zero function $f\in L_M$  and isomorphic to some Orlicz sequence space ${\ell_\psi}$ (see \eqref{psi}), depend not only on degree of "closeness"\;of the function $\psi$ to the function $M$, but also on whether the function $t^{-1/\beta_M^ \infty}$  belongs to the space $L_M$ or not (see \cite{A-22}).

\begin{ex}
\label{ex1}
Let $1<p<2$, $M(u)=u^{p}$ (that is, $L_M=L^p$), $f(t):=t^{-1/p}\ln^{-3/(2p)}(e/t)$, $0<t\le 1$. Then, $f=f^*$, and, if $[f_k]_{L^p}=\ell_\psi$ and $[f_k]_{L^1}=\ell_{\varphi}$, by \eqref{Luxem2}
(see also \cite[Proposition~2.4]{ASZ-15}), we have
\begin{equation}\label{f cond}
\frac1{\psi^{-1}(t)}\asymp \left(\frac1t\int_0^tf(s)^p\,ds\right)^{1/p}+\left(\frac1t\int_t^1f(s)^2\,ds\right)^{1/2},\quad 0<t\le 1,
\end{equation}
\begin{equation}\label{s cond}
\frac1{\varphi^{-1}(t)}\asymp \frac1t\int_0^tf(s)\,ds+\left(\frac1t\int_t^1f(s)^2\,ds\right)^{1/2},\quad 0<t\le 1.
\end{equation}
Now, a combination of standard estimates with integration by parts leads to the following equivalences (the constants of which depend only on $p$):
$$
\frac1t\int_0^tf(s)^p\,ds=\frac1t\int_0^t \ln^{-3/2}(e/s)\,\frac{ds}{s}\asymp \frac{1}{t\ln^{1/2}(e/t)},\;\;0<t\le 1,$$
$$
\frac1t\int_0^tf(s)\,ds=\frac1t\int_0^t s^{-1/p}\ln^{-3/(2p)}(e/s)\,ds\asymp \frac{1}{t^{1/p}\ln^{3/(2p)}(e/t)},\;\;0<t\le 1,$$
and
$$
\frac1t\int_t^1f(s)^2\,ds=\frac1t\int_t^1 s^{-2/p}\ln^{-3/p}(e/s)\,ds\asymp \frac{1}{t^{2/p}\ln^{3/p}(e/t)},\;\;0<t\le 1/2.
$$
Therefore, applying relations \eqref{f cond} and \eqref{s cond}, we obtain 
$$
\psi^{-1}(t)\asymp t^{1/p}\ln^{1/(2p)}(e/t)\;\;\mbox{and}\;\;
\varphi^{-1}(t)\asymp t^{1/p}\ln^{3/(2p)}(e/t),\;\;0<t\le 1.$$
Hence, the functions $\psi$ and $\varphi$ are not equivalent, and hence $\ell_\psi\stackrel{\ne}{\subset}\ell_{\varphi}$. Thus, $[f_k]_{L^p}$ is not a $\Lambda(p)$-subspace.
\end{ex}

In the next example, as in the preceding one, the function $\psi$ is "close"\;to $M$, differing only by a power of the logarithm. However, now $t^{-1/\beta_M^\infty}\in L_M$ (in the example \ref{ex1}, on the contrary, $\beta_M^\infty=p$, and hence $t^{-1/\beta_M^\infty}\not\in L_M=L^p$), and, as a result, the subspace $[f_k]_{L_M}$, isomorphic to the space $\ell_\psi$, is strongly embedded in $L_M$.

\begin{ex}
\label{ex2}
Let $1<p<2$, $0<\alpha<1/p$, $M(u)$ be an Orlicz function equivalent to the function $u^p\ln^{-2} u$ for large values of $u$, $f(t):=t^{-1/p}\ln^{\alpha}(e/t)$, $0<t\le 1$.
Since
$$
\int_0^1M(f(t))\,dt\asymp \int_0^1\ln^{p\alpha-2}(e/t)\,dt/t<\infty,$$
then $f\in L_M$ due to the choice of parameters $p$ and $\alpha$.

Consider an Orlicz function $\psi$ such that $\psi(s)\asymp s^p\ln^{p\alpha}(e/s)$ for small values of the argument. On the one hand, it is immediately verified that $1/\psi^{-1}(t)\asymp f(t)$, $0<t\le 1$. On the other hand, for some $C>0$
$$
\psi(st)\le C\psi(s)\psi(t),\;\;0\le s,t\le 1.$$
Therefore, by \cite[Theorem~4.1]{ASZ-22}, for every symmetric space $X$ such that $f\in X$, we have $[f_k]_X\approx\ell_\psi$, where, as above, $\{f_k\}$ is a sequence of mean zero independent functions, equimeasurable with $f$. In particular,
$[f_k]_{L_M}\approx[f_k]_{L^1}\approx \ell_\psi$, and hence the subspace $[f_k]_{L_M}$ is strongly embedded in $L_M$. Moreover, as we will see in Theorem \ref{cor: new11}, due to the submultiplicativity of $\psi$, the unit ball of this subspace consists of functions having equicontinuous norms in $L_M$.
\end{ex}

In what follows, we will repeatedly use the following statement, which follows from the results of the paper \cite{ASZ-15} on the uniqueness of the distribution of a function whose independent copies span a given subspace in the $L^p$-space.

\begin{lemma}\label{lemma 4dop}
Let $M$ be an Orlicz function, $M\in \Delta_2^\infty$, $f\in L_M$. Suppose that the subspace $[f_k]_{L_M}$, where $\{f_k\}$ is a sequence of independent functions equimeasurable with $f$ and such that $\int_0^1 f_k(t)\,dt =0$, is strongly embedded in $L_M$. Then, if $[f_k]_{L_M}=\ell_\psi$, where $1<\alpha_\psi^0\le\beta_\psi^0<2$, then $n_f(\tau)\asymp n_{ 1/\psi^{-1}}(\tau)$ for large $\tau>0$.
\end{lemma}
\begin{proof}
By the assumption, with constants independent of $n\in\mathbb{N}$ and $a_k\in\mathbb{R}$, we have
$$
\Big\|\sum_{k=1}^n a_kf_k\Big\|_{L_M}\asymp\Big\|\sum_{k=1}^n a_kf_k\Big\|_{L^1}.$$
Furthermore, since $[f_k]_{L_M}\approx \ell_\psi$, then from  \eqref{Luxem} it follows
\begin{equation*}
\frac{1}{\psi^{-1}(1/n)}=\Big\|\sum_{k=1}^n e_k\Big\|_{\ell_\psi}\asymp\Big\|\sum_{k=1}^n f_k\Big\|_{L_M},\;\;n\in\mathbb{N}.
\label{Luxem5}
\end{equation*}
Thus, with constants independent of $n\in\mathbb{N}$, we have
$$
\frac{1}{\psi^{-1}(1/n)}\asymp\Big\|\sum_{k=1}^n f_k\Big\|_{L^1}.$$
Since $1<\alpha_\psi^0\le\beta_\psi^0<2$, then in view of Lemma \ref{Lemma 20}, the function $\psi$ is $(1+\varepsilon)$-convex and $( 2-\varepsilon)$-concave for small values of the argument if $\varepsilon>0$ is sufficiently small.
Consequently, the statement of the lemma is a direct consequence of the last equivalence and Theorem~1.1 from \cite{ASZ-15}, applied in the case of $p=1$.
\end{proof}

\vskip0.2cm

\section{The main results.}
\label{main}

\subsection{A characterization of properties of subspaces generated by independent copies of a mean zero function $f$ in terms of dilations of $f$.\\}
\label{main2}

Let us start with a sufficient (and necessary in many cases) condition, under which a sequence of independent copies of a mean zero  function $f\in L_M$ spans in the given Orlicz space $L_M$ a strongly embedded subspace.

\begin{prop}\label{prop 1a}
Let $M$ be an  Orlicz function, $f\in L_M$. 

(i) if  $\lim_{t\to\infty}M(t)/t=\infty$ and 
\begin{equation}\label{general condition}
\|\sigma_n f\|_{L_M}\preceq \|\sigma_n f\|_{L^1},\;\;n\in\fg,
\end{equation}
then the subspace $[f_k]$ spanned by a sequence of mean zero independent functions $\{f_k\}$, equimeasurable with $f$, is strongly embedded in $L_M$.

(ii) Conversely, if such a sequence $\{f_k\}$ as in (i) spans in $L_M$ a strongly embedded subspace, isomorphic to an Orlicz space
${\ell_\psi}$, with $1<\alpha_\psi^0\le\beta_\psi^0<2$,
then inequality \eqref{general condition} holds.
\end{prop}
\begin{proof}
Without loss of generality, assume that $f=f^*$.

(i) According to the discussion in \S\,\ref{aux1}, the sequence $\{f_k\}$ is equivalent in the space $L_M$ (resp. $L^1$) to the canonical basis in some Orlicz sequence space $ {\ell_\psi}$ (resp. $\ell_{\theta}$). Since $\lim_{t\to\infty}M(t)/t=\infty$, then $L_M\ne L^1$. Consequently, by Lemma \ref{lemma 2dop}, it suffices to show that ${\ell_\psi}= \ell_{\theta}$, or it is the same, that the fundamental functions of these spaces are equivalent for small $t>0$ (see \S\,\ref{prel2}). Since, due to \eqref{Luxem2},
\begin{equation}
\label{Luxem2e}
\frac{1}{\psi^{-1}(1/n)}\asymp \|\sigma_nf\|_{L_M}+\Big(n\int_{1/n}^1f(s)^2\,ds\Big)^{1/2},\;\;n\in\mathbb{N},
\end{equation}
and similarly
\begin{equation*}
\frac{1}{\theta^{-1}(1/n)}\asymp \|\sigma_nf\|_{L^1}+\Big(n\int_{1/n}^1f(s)^2\,ds\Big)^{1/2},\;\;n\in\mathbb{N},
\end{equation*}
then the required equivalence follows from condition \eqref{general condition}, formula \eqref{fund} for the fundamental function of an Orlicz space, as well as from the convexity of $\psi$ and ${\theta}$.

(ii) It suffices to show that inequality \eqref{general condition} holds for all $n$ sufficiently large.

Since $\psi^{-1}$ is an increasing, concave function on $(0,1]$, then $\psi^{-1}(t)\le\psi^{-1}(Ct )\le C\psi^{-1}(t)$ for any $C\ge 1$ and all $0<t\le 1$, and also the function $1/\psi^{-1}$ coincides with its non-increasing rearrangement. Moreover, by Lemma \ref{lemma 4dop}, the distribution functions $n_f(\tau)$ and $n_{1/\psi^{-1}}(\tau)$ are equivalent for large $\tau>0$. Combining this together with the definition of the non-increasing rearrangement of a measurable function (see \S\,\ref{prel1}), for some $t_0\in (0,1]$, we get
$$
f(t)\asymp 1/\psi^{-1}(t),\;\;0<t\le t_0.$$
Thus, since \eqref{Luxem2e} is satisfied by the assumtion, for a sufficiently large $n_0\in\mathbb{N}$ we obtain that
$$
\|\sigma_nf\|_{L_M}\preceq f(1/n),\;\; n\ge n_0.$$
Now, inequality \eqref{general condition} for $n\ge n_0$ is a direct consequence of the last estimate and the inequality
$$
f(1/n)\le n\int_0^{1/n} f(u)\,du=\int_0^{1} f(u/n)\,du=\|\sigma_nf\|_{L^1},\;\;n\in\mathbb{N}.$$
This completes the proof.
\end{proof}

In the same terms, we can state also a condition of equicontinuity of the $L_M$-norms of functions of the unit ball of such a subspace of  $L_M$.

\begin{prop}\label{prop 2a}
Let $M$ be an Orlicz function, $\lim_{t\to\infty}M(t)/t=\infty$, $f\in L_M$  and let $\{f_k\}$ be a sequence of mean zero independent functions, equimeasurable with $f$. Consider the following conditions:

(a) the unit ball of the subspace $[f_k]$ consists of functions with equicontinuous norms in $L_M$;

(b) there is a convex, non-decreasing function $N$ on $[0,\infty)$ such that $N(0)=0,$ $N\in\Delta_2^\infty$, $\lim_{u\to \infty}{N(u)}/{M(u)}=\infty$ and
\begin{equation}\label{general condition2}
\|\sigma_n f\|_{L_N}\preceq\|\sigma_n f\|_{L_M},\;\;n\in\fg.
\end{equation}

We have $(b)\Longrightarrow (a)$. If additionally one has $[f_k]_{L_M}\approx {\ell_\psi}$, where $1<\alpha_\psi^0\le\beta_\psi^0<2$, then the inverse implication $(a)\Longrightarrow (b)$ holds as well.
\end{prop}
\begin{proof}
$(b)\Longrightarrow (a)$. 
First, from \eqref{general condition2} and the condition $f\in L_M$ it follows that $f\in L_N$. Moreover, since
$$
\lim_{u\to\infty}\frac{M(u)}{u}=\lim_{u\to\infty}\frac{N(u)}{u}=\infty,$$ 
then, arguing exactly in the same way as in the proof of Proposition \ref{prop 1a}(i), we can show that the sequence $\{f_k\}$ in both spaces $L_M$ and $L_N$ is equivalent to the canonical basis in the same Orlicz sequence space.
Hence, the norms of these spaces are equivalent on the subspace $H:=[f_k]_{L_M}$, i.e., for some $C>0$
\begin{equation}\label{embedding2}
B_H\subset \{x\in L_N:\,\|x\|_{L_N}\le C\}
\end{equation}
Moreover, due to the conditions and Lemma~3 from \cite{A-14}, we infer that the embedding $L_N\subset L_M$ is {\it strict}. This means that
$$
\lim_{\delta\to 0}\sup_{\|x\|_{L_N}\le 1,m({\rm supp\,}x)\le\delta}\|x\|_{L_M}=0$$
(for more details related to properties of strict embeddings of symmetric spaces, see \cite{AS-19}). As a result we get
$$
\lim_{\delta\to 0}\sup_{x\in B_H,m({\rm supp\,}x)\le\delta}\|x\|_{L_M}=0,$$
and $(a)$ follows.

$(a)\Longrightarrow (b)$. 
Let $H:=[f_k]$. According to the condition and Vall\'{e}e Poussin's criterion (see, e.g., Theorem~3.2 from \cite{LMT}), there exists a nondecreasing convex function $Q$ on $[0,\infty)$ such that $Q(0)=0 $, $Q\in\Delta_2^\infty$, $\lim_{u\to\infty}{Q(u)}/{u}=\infty$ and $\sup_{x\in B_H}\|Q (|x|)\|_{L_M}<\infty$. The last relation means that, for some $C\ge 1$ and all $x\in B_H$, we have
$$
\int_0^1 M\Big(\frac{Q(|x(t)|)}{C}\Big)\,dt\le 1.$$
Since the function $Q$ is convex, then $Q(|x(t)|)/C\ge Q(|x(t)|/C)$, whence
$$
\int_0^1 M\Big(Q\Big(\frac{|x(t)|}{C}\Big)\Big)\,dt\le 1$$
for all $x\in B_H$. Setting $N(u):=M(Q(u))$ and taking into account the properties of the functions $M$ and $Q$, it is easy to verify that the function $N$ satisfies all the conditions listed in (b). In addition, due to the last inequality, embedding \eqref{embedding2} is still valid. Thus, the $L_M$- and $L_N$-norms are equivalent on the subspace $H$. Since $H$ is strongly embedded in $L_M$ by condition and Lemma \ref{prop 1ab}, then $H$ is strongly embedded in $L_N$ as well (see also Lemma \ref{lemma 2dop}). Therefore, applying Proposition \ref{prop 1a}(ii), we get
$$
\|\sigma_n f\|_{L_N}\preceq \|\sigma_n f\|_{L^1}\le\|\sigma_n f\|_{L_M},\;\;n\in\fg.
$$
As a result, inequality \eqref{general condition2} and hence the proposition are proved.
\end{proof}

\subsection{Subspaces of $L_M$ spanned by independent copies of mean zero functions, whose the unit ball consists of functions with equicontinuous  $L_M$-norms.\\}
\label{main3}

Let $h:\,[0,1]\to [0,\infty)$, $h(t)>0$ if $0<t\le 1$. Recall that the dilation function ${\mathcal M}_h$ of $h$ is defined as follows:
$$
{\mathcal M}_h(t):=\sup_{0<s\le \min(1,1/t)}\frac{h(st)}{h(s)},\;\;t>0.$$

\begin{prop}\label{prop 1dop}
Suppose $\psi:\,[0,1]\to [0,1]$ is an increasing and continuous function, $\psi(0)=0$, $\psi(1)=1$.
If $f(t):=1/\psi^{-1}(t)$, $0<t\le 1$, and $g$ is a nonincreasing, nonnegative function on $(0,1]$ such that
\begin{equation}\label{distribution2}
n_g(\tau) = \min\left({\mathcal{M}}_\psi\left(\frac{1}{\tau}\right),1\right), \tau>0,
\end{equation}
then for any sequence $c=(c_k)\in {\ell_\psi}$ the following inequality holds:
$$
(c\bar\otimes f)^*\cdot\chi_{(0,1)}\le \|c\|_{{\ell_\psi}}g.$$
\end{prop}
\begin{proof}
Without loss of generality, we will further assume that $\|c\|_{\ell_\psi}=1.$

Firstly, we observe that, thanks to the properties of $\psi$, the function from the right-hand side of equality \eqref{distribution2} is nonnegative, continuous and nonincreasing. Moreover, it does not exceed $1$ and tends to zero as $\tau$ tends to infinity. Therefore, there exists a nonincreasing function $g:\,(0,1]\to [0,\infty)$ that satisfies \eqref{distribution2}.

Since $\psi$ does not decrease and $\psi(0)=0$, we have for each $\tau\ge 1$:
$$
n_f(\tau) = m\left\{ u\in (0,1] :\, \frac{1}{\psi^{-1}(u)} > \tau\right\} = m\left\{ u\in (0,1] :\, \psi \left(\frac{1}{\tau}\right) > u\right\}= \psi\left(\frac{1}{\tau}\right).
$$	
Therefore, in view of the definition of the function $c \bar\otimes f$ (see \S\,\ref{aux1}),
\begin{equation}\label{distribution3}
n_{c \bar\otimes f} (\tau) = \sum_{k=1}^\infty n_{c_k f} (\tau) = \sum_{k=1}^\infty \psi\left(\frac{|c_k|}{\tau}\right).
\end{equation}

In addition, since $\|c\|_{\ell_\psi} = 1$, then for any $k=1,2,\dots$
$$
\psi(|c_k|) \leq \sum_{i=1}^\infty \psi(|c_i|)=1= \psi(1).$$
Taking into account the monotonicity of $\psi$ once more, we obtain then that $|c_k| \leq 1$ for all $k=1,2,\dots$. Hence, by the definition of the function ${\mathcal{M}}_\psi$, we have  for each $\tau \ge 1$ and all $k=1,2,\dots$:
$$
\psi\left(\frac{|c_k|}{\tau}\right)\leq \psi(|c_k|){\mathcal{M}}_\psi\left(\frac{1}{\tau}\right).$$
Thus, since $\|c\|_{\ell_\psi}=1$ and $\psi$ increases, from \eqref{distribution2} and \eqref{distribution3} it follows
\begin{eqnarray}
n_{c \bar\otimes f} (\tau) \leq {\mathcal{M}}_\psi\left(\frac{1}{\tau}\right) \sum_{k=1}^\infty \psi(|c_k|)\leq {\mathcal{M}}_\psi\left(\frac{1}{\tau}\right)=n_g(\tau),\;\;\tau\ge 1.
\label{distr ineq}
\end{eqnarray}

Now, let us check that for each $s \in (0,1)$ it holds
\begin{equation}
\label{embedd}
\{ \tau>0:\, n_g(\tau)\leq s\}\subset \{ \tau>0:\,n_{c \bar\otimes f} (\tau)\leq s\}.
\end{equation}
Indeed, $n_g(1)={\mathcal{M}}_\psi(1)=1$, whence $g(t)> 1$ a.e. on $(0,1]$. Hence, 
$$
\{ \tau>0:\, n_{g}(\tau)\leq s\}\subset (1,\infty),$$ 
and therefore, by \eqref{distr ineq}, the inequality $n_{g}(\tau)\leq s$ implies that $n_{c \bar\otimes f}(\tau)\leq s.$
Thus, the embedding \eqref{embedd} is proved.

Since $g$ does not increase, then, by the definition of the nonincreasing rearrangement, from \eqref{embedd} it follows
$$
(c\bar\otimes f)^*\cdot \chi_{(0,1)} \leq g,$$
which completes the proof.
\end{proof}

\begin{rem}
Suppose that the function ${\mathcal{M}}_\psi(t)$ strictly increases on $(0,1]$. Then, as is easy to check, the function $g$ defined by  \eqref{distribution2} coincides with the inverse function ${\mathcal{M}}_\psi^{-1}(t)$.
\end{rem}

From Proposition \ref{prop 1dop} and the definition of a symmetric space it follows

\begin{cor}
\label{cor: new10}
Let $\psi:\,[0,1]\to [0,1]$ be an increasing, continuous function, $\psi(0)=0$, $\psi(1)=1$, $f(t): =1/\psi^{-1}(t)$, $0<t\le 1$, and $g$ be a nonincreasing, nonnegative function on $(0,1]$ such that
its distribution function $n_g(\tau)$ is defined by  \eqref{distribution2}. If $X$ is a symmetric space on $[0,1]$ such that $g\in X$, then for any sequence $c=(c_k)\in {\ell_\psi}$ we have
$$
\|(c\otimes f)^*\cdot\chi_{(0,1)}\|_X\le \|g\|_X\|c\|_{{\ell_\psi}}.$$
\end{cor}

Further, we will need the following technical lemma.

\begin{lemma}\label{lemma 3dop}
If a function $\psi:\,[0,1]\to [0,1]$ increases, $\psi(0)=0$, $\psi(1)=1$, and $h(t)={\mathcal M}_{1/\psi^{-1}}(t)$, $0<t\le 1$,  then
$$
n_h(\tau)\ge \min\left({\mathcal{M}}_\psi\left(\frac{1}{\tau}\right),1\right), \tau>0.$$
\end{lemma}
\begin{proof}
Since $\psi$ increases, ${\mathcal{M}}_\psi(1)=1$, and $h$ does not increase, it suffices to show that for any $\tau\ge 1$ and arbitrarily small $\varepsilon>0$ it holds
\begin{equation}
\label{tech in}
h\left({\mathcal{M}}_\psi\left(\frac{1}{\tau}\right)-\varepsilon\right)>\tau.
\end{equation}

Denote $t:={\mathcal{M}}_\psi({1}/{\tau})-\varepsilon$. By the definition of $h$, we have 
$$
h(t)=\sup_{0<s\le 1}\frac{\psi^{-1}(s)}{\psi^{-1}(st)}=\sup_{0<u\le t\le 1}\frac{\psi^{-1}(u/t)}{\psi^{-1}(u)}.$$
Thus, \eqref{tech in} holds if and only if there is $u>0$ such that $0<u\le t\le 1$ and
$$
\psi^{-1}(u/t)>\tau \psi^{-1}(u),$$ 
or equivalently
$$
u>t\psi(\tau \psi^{-1}(u)).$$
Note that $\tau \psi^{-1}(u)\le 1$. Therefore, after changing $\psi^{-1}(u)=v$ we obtain that the last inequality is valid if and only if
$$
{\mathcal{M}}_\psi\left(\frac{1}{\tau}\right):=\sup_{0<v\le 1}\frac{\psi(v/\tau)}{ \psi(v)}>t.$$
Since the choice of $t$ ensures that the latter is true, inequality  \eqref{tech in} and hence the lemma are proved.
\end{proof}

\begin{theor}
\label{theorem-main4}
Let $M$ be an Orlicz function such that $1<\alpha_M^\infty\le \beta_M^\infty<2$. Assume also that $f\in L_M$ and ${\mathcal{M}}_{f^*}\in L_M$. Then, if $\{f_k\}$ is a sequence of mean zero independent functions equimeasurable with $f$ and $[f_k]_{L_M }\approx\ell_\psi$, where $1<\alpha_\psi^0\le \beta_\psi^0<2$, then the unit ball of the subspace $[f_k]_{L_M}$ consists of functions with equicontinuous norms in $L_M $.
\end{theor}
\begin{proof}
Without loss of generality, we can assume that $f=f^*$.
Let us first prove that the subspace $[f_k]_{L_M}$ is strongly embedded in $L_M$.

From the definition of the dilation function ${\mathcal{M}}_f$ it follows 
$$
\sigma_{1/s}f(t)=f(st)\le {\mathcal{M}}_f (t)f(s),\;\;0<s,t\le 1.$$
Since ${\mathcal{M}}_f\in L_M$ by condition and $f$ is a nonnegative, nonincreasing function, then this inequality implies that for all $0<s\le 1$
$$
\|\sigma_{1/s}f\|_{L_M}\le \|{\mathcal{M}}_f\|_{L_M}f(s)\le \|{\mathcal{M}}_f\|_{L_M}\cdot \frac1s\int_0^s f(u)\,du=\|{\mathcal{M}}_f\|_{L_M}\|\sigma_{1/s}f\|_{L^1}.$$
Thus, applying Proposition \ref{prop 1a}(i), we obtain the required result.

Let us now proceed with the proof of the theorem. Since the subspace $[f_k]_{L_M}$ is strongly embedded in $L_M$, by Lemma \ref{lemma 4dop}, we have $n_f(\tau)\asymp n_{1/\psi^{-1}}(\tau)$ for large $\tau>0$. Therefore, since the functions $f$ and $1/\psi^{-1}$ do not increase and $\psi^{-1}(1)=1$, just as in the proof of Proposition 
\ref{prop 1a}(ii), for some $0<t_0\le 1$ we get
$$
f(t)\asymp 1/\psi^{-1}(t),\;0<t\le t_0,\;\;\mbox{and}\;\;f(t)\preceq 1/\psi^{-1}(t),\;0<t\le 1.$$
Consequently, 
$$
{\mathcal M}_{1/\psi^{-1}}(t)=\sup_{0<s\le 1}\frac{\psi^{-1}(s)}{\psi^{-1}(st)}\preceq \sup_{0<s\le 1}\frac{f(st)\psi^{-1}(s)f(s)}{f(s)}\preceq {\mathcal M}_{f}(t),\;\;0<t\le t_0.$$
Since the function ${\mathcal M}_{1/\psi^{-1}}$ does not increase and,  by condition, ${\mathcal M}_{f}\in L_M$, then from the latter inequality, Lemma \ref {lemma 3dop} and the definition of the function $g$ (see Proposition \ref{prop 1dop}) it follows that
$g$ belongs to the space $L_M$.

Next, by using Lemma \ref{lemma 3}, we find a function $N$ equivalent to some Orlicz function such that $N(1)=1,$ $N\in \Delta_2^\infty$, $\tilde{N} \in \Delta_2^\infty$, $\lim_{u\to\infty}{N(u)}/{M(u)}=\infty$ and $g\in L_N$. Assuming that $N$ is an Orlicz function itself, according to Corollary \ref{cor: new10}, we obtain for any sequence $c=(c_k)\in {\ell_\psi}$ 
$$
\|(c\bar\otimes f)^*\cdot\chi_{(0,1)}\|_{L_N}\le \|g\|_{L_N}\|c\|_{{\ell_\psi}}.$$
Since (see \S\,\ref{aux1})
$$
\Big\|\sum_{k=1}^\infty c_kf_k\Big\|_{L_N}\asymp \|(c\bar\otimes f)^*\cdot\chi_{(0,1)}\|_{L_N}+\|(c\bar\otimes f)^*\cdot\chi_{(1,\infty)}\|_{L^2},$$
$$
\Big\|\sum_{k=1}^\infty c_kf_k\Big\|_{L_M}\asymp \|(c\otimes f)^*\cdot\chi_{(0,1)}\|_{L_M}+\|(c\otimes f)^*\cdot\chi_{(1,\infty)}\|_{L^2}\asymp\|c\|_{{\ell_\psi}}$$
and $L_N\subset L_M$, this implies that 
$$
\Big\|\sum_{k=1}^\infty c_kf_k\Big\|_{L_N}\asymp \|c\|_{{\ell_\psi}}.$$
As a result, to complete the proof it suffices to apply Vall\'{e}e Poussin's criterion \cite[Theorem~3.2]{LMT}.
\end{proof}

The next theorem gives simple sufficient conditions, under which  the unit ball of a strongly embedded subspace of $L_M$ spanned by independent copies of a mean zero function from $L_M$ consists of functions having equicontinuous norms in $L_M$.

\begin{theor}
\label{cor: new11}
Let $M$ be an Orlicz function such that $1<\alpha_M^\infty\le \beta_M^\infty<2$. Suppose that $\{f_k\}$ is a sequence of mean zero independent functions equimeasurable with a function $f\in L_M$ and $[f_k]\approx \ell_\psi$, where $1<\alpha_\psi^0\le \beta_\psi^0<2$. Assume also that the subspace $[f_k]$ is strongly embedded in $L_M$.

If there exists a function $\varphi\in C_{\psi,1}^0$ such that for some $C>0$ and all $s,t\in [0,1]$
\begin{equation}
\label{main submult}
\psi(st)\le C\psi(s)\varphi(t),
\end{equation}
then the unit ball of the subspace $[f_k]$ consists of functions with equicontinuous norms in $L_M$. In particular, this holds if at least one of the following conditions is fulfilled:

(a) $\psi$ is submultiplicative, i.e., there exists $C>0$ such that for all $s,t\in [0,1]$
$$
\psi(st)\le C\psi(s)\psi(t);$$

(b) $\psi$ is equivalent to some $\alpha_\psi^0$-convex function for small values of the argument;

(c) $t^{-1/p}\in L_M$ for some $p\in (0,\alpha_\psi^0)$.
\end{theor}

\begin{proof}
It is obvious that inequality \eqref{main submult} holds if and only if
\begin{equation}
\label{tech}
\psi^{-1}(t/s){\varphi^{-1}(s)}\le C_1 {\psi^{-1}(t)}
\end{equation}
for some $C_1>0$ and all $0<t\le s\leq 1$. 
Hence,
$$
{\mathcal{M}}_{\psi^{-1}}(1/s)=\sup_{0\le t\le s}\frac{\psi^{-1}(t/s)}{\psi^{-1}(t)}\le C_1\cdot \frac{1}{\varphi^{-1}(s)},\;\;0<s\le 1.$$
Since the subspace $[f_k]$ is strongly embedded in $L_M$, $[f_k]\approx\ell_\psi$ and $\varphi\in C_{\psi,1}^0$, by Lemma 
\ref{prop1}, the function $1/\varphi^{-1}$ belongs to the space $L_M$. Therefore, from the latter inequality it follows that ${\mathcal{M}}_{\psi^{-1}}(1/s)\in L_M$.

On the other hand, by Lemma \ref{lemma 4dop}, the distribution functions $n_f(\tau)$ and $n_{1/\psi^{-1}}(\tau)$ are equivalent for large $\tau>0$. Therefore, as above, the functions $f^*(t)$ and $1/\psi^{-1}(t)$ are equivalent for small $t>0$, and, thanks to the equality $\psi^{-1} (1)=1$, we obtain that for some $C>0$ and all $0<s\leq 1$
\begin{equation}
\label{tech1}
{\mathcal{M}}_{f^*}(s)\le C{\mathcal{M}}_{1/\psi^{-1}}(s)= C{\mathcal{M}}_{\psi^{-1}}(1/s).
\end{equation}
Thus, ${\mathcal{M}}_{f^*}\in L_M$ and for completing the proof of the first statement of the theorem, it remains to apply Theorem \ref{theorem-main4}. Let us show that the left statements of the theorem are   consequences of the first one.

Indeed, to get the result in the case $(a)$, we need only to note that $\psi\in C_{\psi,1}^0$ by the assumption. Further, according to Lemma \ref{Lemma 20}, the function $\psi$ is equivalent to some $p$-convex function for small values of the argument if and only if
for some $C_2>0$ and all $0<t,s\leq 1$ the following holds:
\begin{equation}
\label{tech2}
\psi(st)\le C_2 s^{p}\psi(t).
\end{equation}
Therefore, if $(b)$ is satisfied, then the desired statement is an immediate consequence of the fact that the function
$t^{\alpha_\psi^0}$ belongs to the set $C_{\psi,1}^0$ (see \S\,\ref{prel2}).

Finally, in view of the definition of the index $\alpha_\psi^0$, for each $p\in (0,\alpha_\psi^0)$ the function $\psi$ is equivalent to some $p$-convex function for small values of the argument, whence for such $p$ the inequality \eqref{tech2} holds. Thus, the desired result follows from the condition $(c)$, and hence the proof of the theorem is completed.
\end{proof}

\begin{rem}
In general, Theorem \ref{cor: new11} cannot be extended to the whole class of subspaces of an Orlicz space $L_M$ that are isomorphic to some  Orlicz sequence spaces. As it is shown in \cite{A-22} (see Theorem~2 and its proof), if the function $t^{-1/\beta_M^\infty}\in L_M$, then $L_M$ contains a strongly embedded subspace $H$ of such a type, whose unit ball consists of functions with non-equicontinuous norms in $L_M$.
\end{rem}

\vskip0.2cm

\subsection{Subspaces of Orlicz spaces generated by mean zero identically distributed  independent functions and Matuszewska-Orlicz indices.}
\label{main1}
In the case, when $t^{-1/\beta_M^\infty}\not\in L_M$ (in particular, this condition is satisfied by $L^p$), all subspaces under consideration, which are strongly embedded in the Orlicz space $L_M$, can be characterized by using the Matuszewska-Orlicz indices of the corresponding functions. Moreover, the same condition is equivalent to the fact that the unit ball of such a subspace consists of functions with equicontinuous $L_M$-norms.

\begin{theor}
\label{new theorem}
Let $M$ be an Orlicz function such that $1<\alpha_M^\infty\le \beta_M^\infty<2$ and $t^{-1/\beta_M^\infty}\not\in L_M$. If $f\in L_M$ and $\{f_k\}$ is a sequence of mean zero independent functions equimeasurable with $f$, then the following conditions are equivalent:

(a) the unit ball of the subspace $[f_k]$ consists of functions with  equicontinuous norms in $L_M$;

(b) the subspace $[f_k]$ is strongly embedded in $L_M$;

(c) $\alpha_\psi^0>\beta_M^\infty$, where the Orlicz function $\psi$ is such that $[f_k]_{L_M}\approx\ell_\psi$.
\end{theor}
\begin{proof}
As above, we can assume that $f=f^*$.

The implication $(a)\Rightarrow (b)$ is a consequence of Lemma \ref{prop 1ab}. As for the implication $(b)\Rightarrow (c)$, it is obvious if $\alpha_\psi^0\ge 2$. In the case when $\alpha_\psi^0<2$, it follows from Lemma \ref{prop1} (see also its proof). Thus, it remains only to show that $(c)$ implies $(a)$.

So, let $\alpha_\psi^0>\beta_M^\infty$. Also, assume that $p\in (\beta_M^\infty,\alpha_\psi^0)$. Then, from the definition of the index $\beta_M^\infty$ it follows  
\begin{equation}
\label{eq-new0}
\lim_{u\to\infty}\frac{u^p}{M(u)}=\infty.
\end{equation}

To prove statement $(a)$ it suffices to show that the norms of the spaces $L_M$ and $L^p$ are equivalent on $H$, or it is the same, to check that $f\in L^p$ and $[f_k]_{L^p }\approx\ell_\psi$. Indeed, then the unit ball $B_H$ of the subspace $H:=[f_k]_{L_M}$ is bounded in $L^p$, and therefore, by \eqref{eq-new0}, according to the Vall\'{e}e Poussin criterion (see, e.g., \cite[theorem~ 3.2]{LMT}), the set $B_H$ consists of functions having equicontinuous norms in $L_M$, i.e., $(a)$ is done.

First of all, due to the inequality $\alpha_\psi^0>p$ and Lemma \ref{Lemma 20}, the function $\psi$ is equivalent to some 
$(p+\varepsilon)$-convex function for small values of the argument whenever $\varepsilon>0$ is sufficiently small. Therefore, ${1}/{\psi^{-1}}\in L^p$ and, applying \cite[Proposition~3.1]{ASZ-15}, we obtain that
$$
\|\sigma_{1/t}(1/\psi^{-1})\|_{L^p}=\Big(\frac1t\int_0^t \frac{ds}{(\psi^{-1}(s))^p}\Big)^{1/p}\preceq \frac{1}{\psi^{-1}(t)},\;\;0<t\le 1.$$
Next, since $f(t)$ does not increase, $L_M\subset L^1$ and $[f_k]_{L_M}\approx\ell_\psi$, from \eqref{Luxem2} it follows
$$
f(t)\le \frac1t\int_0^t f(s)\,ds=\|\sigma_{1/t}f\|_{L^1}\le\|\sigma_{1/t}f\|_{L_M}\preceq  \frac{1}{\psi^{-1}(t)},\;\;0<t\le 1.$$
Therefore, in particular, $f\in L^p$. In addition, the last relations and \eqref{Luxem2} imply the estimate
\begin{eqnarray*}
\|\sigma_nf\|_{L^p}+\Big(n\int_{1/n}^1f(s)^2\,ds\Big)^{1/2}&\preceq& \|\sigma_n(1/\psi^{-1})\|_{L^p}+\Big(n\int_{1/n}^1f(s)^2\,ds\Big)^{1/2}\\&\preceq& \frac{1}{\psi^{-1}(1/n)},\;\;n\in\mathbb{N}.
\end{eqnarray*}
In view of the embedding $L^p\subset L_M$ and relation \eqref{Luxem2} once again, we obtain also the opposite inequality, i.e.,
\begin{equation*}
\frac{1}{\psi^{-1}(1/n)}\asymp \|\sigma_nf\|_{L^p}+\Big(n\int_{1/n}^1f(s)^2\,ds\Big)^{1/2},\;\;n\in\mathbb{N}.
\end{equation*}
Thus, $[f_k]_{L^p}\approx \ell_\psi$, and the theorem is proved.

\end{proof}

\begin{rem}
The condition $t^{-1/\beta_M^\infty}\not\in L_M$ is used only in the proof of the implication $(b)\Rightarrow (c)$ (when applying Lemma \ref{prop1}). Hence, the implication $(c)\Rightarrow (a)$ holds for any Orlicz space $L_M$ such that $1<\alpha_M^\infty\le \beta_M^\infty<2$.
\end{rem}

\begin{rem}
Let us assume that an Orlicz function $M$ satisfies the conditions of Theorem \ref{new theorem}. As proven in \cite[Theorem~3]{A-22}, the equivalence of the conditions $(a)$ and $(b)$ holds for {\it all} subspaces of $L_M$, which are isomorphic to Orlicz sequence spaces.
\end{rem}

In particular, for $L^p$-spaces from the last theorem and its proof we get the following complement to Rosenthal's theorem (see 
\S\,\ref{Intro}). 

\begin{cor}
\label{cor: new1}
Let $1<p<2$, $f\in L^p$ and $\{f_k\}$ be a sequence of mean zero independent functions equimeasurable with $f$ such that $[f_k]_{L^p}\approx\ell_\psi$. The following conditions are equivalent:

(a) $[f_k]_{L^p}$ is a $\Lambda(p)$-space;

(b)  $[f_k]_{L^p}$ is a $\Lambda(q)$-space for some $q>p$;

(c) $\alpha_\psi^0>p$.
\end{cor}

\subsection{Subspaces of $L^2$ spanned by independent copies of a mean zero function $f\in L^2$.\\}
\label{main4}

So far we have considered subspaces of Orlicz spaces $L_M$,
lying "strictly to the left"\:of the space $L^2$, or more precisely, such that $1<\alpha_M^\infty\le \beta_M^\infty<2$. The following result shows that in the case when $M(t)=t^2$ (i.e., in $L^2$), the situation is much simpler: the unit ball {\it of any} subspace of $L^2$ spanned by mean zero identically distributed independent functions consists of functions with equicontinuous $L^2$-norms.

\begin{theor}\label{prop 5}
Let $\{f_k\}_{k=1}^\infty$ be a sequence of mean zero independent functions equimeasurable with some function $f\in L^2$. Then, the unit ball $B_H$ of the subspace $H:=[f_k]_{L^2}$ consists of functions having  equicontinuous norms in $L^2$.
\end{theor}
\begin{proof}
As usual, we assume that $f^*=f$.

By Lemma \ref{lemma 3}, we find a function $N$ equivalent to some $2$-convex Orlicz function such that $\tilde{N}\in\Delta_2^\infty$, $\lim_{u\to\infty }{N(u)}u^{-2}=\infty$ and $N(|f|)\in L^1$. Without loss of generality, we can assume that $N$ is itself a $2$-convex Orlicz function on $[0,\infty)$, and therefore the Orlicz space $L_N$ is $2$-convex (see \S\,\ref{prel2}). In addition, from the above relations it follows that $L_N\stackrel{\ne}{\subset} L^2$ and $f\in L_N.$

Let $[f_k]_{L_N}\approx \ell_\psi$ and let $\phi_{\ell_\psi}$ be the fundamental function of the space $\ell_\psi$. By virtue of \eqref{Luxem2} and the definition of the operator $\sigma_n$, for any $n\in\mathbb{N}$ we have
\begin{eqnarray*}
\phi_{\ell_\psi}(n)&\asymp& \|\sigma_nf\|_{L_N}+\|f(\cdot/n)\|_{L^2[1,\infty)}\\&=&
\|\sigma_n(f\chi_{[0,1/n]})\|_{L_N}+\|f\chi_{[1/n,1]}(\cdot/n)\|_{L^2[1,\infty)}\\&\le&
C'n^{1/2}(\|f\chi_{[0,1/n]}\|_{L_N}+\|f\chi_{[1/n,1]}\|_{L^2})\le Cn^{1/2}\|f\|_{L_N}.
\end{eqnarray*}

On the other hand, $\{f_k/\|f\|_{L^2}\}_{k=1}^\infty$ is an orthonormal sequence in $L^2$ and hence $[f_k] _{L^2}\approx \ell^2$. Since $\ell_\psi\subset \ell^2$ and $\phi_{\ell^2}(n)=n^{1/2}$, $n=1,2,\dots$, from the preceding relations it follows that $\phi_{\ell_\psi}(n)\asymp n^{1/2}$, i.e., $[f_k]_{L_N}\approx \ell^2$.
Thus, the ball $B_H$ is bounded in $L_N$, and applying the Vall\'{e}e Poussin criterion once again, we obtain the desired result.
\end{proof}

\begin{rem}
\label{rem: non-equi-int}
The following example shows that the result of Theorem \ref{prop 5} cannot be extended to all subspaces generated by mean zero independent (but. in general, not identically distributed) functions.

Let $\{f_k\}_{k=1}^\infty$ be a sequence of independent functions on $[0,1]$ such that $\int_0^1 f_k(t)\,dt=0,$
$|f_k(t)|=2^{k/2}$, $t\in E_k,$ where $m(E_k)=2^{-k-1},$ and $|f_k(t)|= 1$, $t\in [0,1]\setminus E_k$, $k=1,2,\dots.$ As it is shown in \cite[example~2]{A-14}, the subspace $[f_k] $ is strongly embedded in $L^2$, but there is no symmetric space $X$ such that both $X\stackrel{\ne}{\subset} L^2$ and $X\supset [f_k]$. Then, taking into account the Vall\'{e}e Poussin criterion, we conclude that the norms of functions of the unit ball of the subspace $[f_k]$ are not equicontinuous in $L^2$.
\end{rem}

\end{document}